\documentclass[12pt]{amsart}

\usepackage{amssymb,a4wide,graphicx,hyperref}

\newtheorem*{theorem}{Main theorem}
\newtheorem{corollary}{Corollary}
\newtheorem{lemma}{Lemma}[section]
\theoremstyle{remark}
\newtheorem{remark}{Remark}[section]

\newcommand\decdot{\raisebox{0.1ex}{\textbf{.}}}

\DeclareMathOperator{\e}{e}

\title{Patterns in rational base number systems}

\author[J.~F.~Morgenbesser]{Johannes F. Morgenbesser}
\address{Universit\"at Wien, Fakult\"at f\"ur Mathematik, Nordbergstrasse 15, 1090 Wien, AUSTRIA}
\email{johannes.morgenbesser@univie.ac.at}
\author[W.~Steiner]{Wolfgang Steiner}
\address{LIAFA, CNRS UMR 7089, Universit\'e Paris Diderot -- Paris 7,
Case 7014, 75205 Paris Cedex 13, FRANCE}
\email{steiner@liafa.jussieu.fr}
\author[J. M. Thuswaldner]{J\"org M. Thuswaldner}
\address{Department of Mathematics, 275 TMCB, Brigham Young University, Provo, UT 84602, USA}
\email{joerg.thuswaldner@unileoben.ac.at}

\subjclass[2010]{11A63 (primary), 28A80, 52C22 (secondary)}

\keywords{Rational number system, normal numbers, $p$-adic numbers, Fourier analysis, sum-of-digits function}

\thanks{This research was supported by the Austrian Science Fund (FWF), projects P21209, S9610, and W1230. Part of this research was conducted while the second author was visiting academic at the Department of Computing of the Macquarie University, Sydney.}

\begin{document}

\begin{abstract}
Number systems with a rational number $a/b > 1$ as base have gained interest in recent years. In particular, relations to Mahler's $\frac32$-problem as well as the Josephus problem have been established. In the present paper we show that the patterns of digits in the representations of positive integers in such a number system are uniformly distributed. We study the sum-of-digits function of number systems with rational base $a/b$ and use representations w.r.t.\ this base to construct normal numbers in base $a$ in the spirit of Champernowne.

The main challenge in our proofs comes from the fact that the language of the representations of integers in these number systems is not context-free. The intricacy of this language makes it impossible to prove our results along classical lines. In particular, we use self-affine tiles that are defined in certain subrings of the ad\`ele ring $\mathbb{A}_\mathbb{Q}$ and Fourier analysis in $\mathbb{A}_\mathbb{Q}$. With help of these tools we are able to reformulate our results as estimation problems for character sums.
\end{abstract}

\maketitle

\section{Introduction}

Starting with the well-known papers by Gelfond~\cite{Gelfond:68} and Delange~\cite{D75}, distribution properties of sets defined in terms of digital restrictions and sum-of-digits functions have been studied systematically by many authors. Recently, Mauduit and Rivat~\cite{MR:10} solved a problem on the distribution of the $q$-ary sum-of-digits function of primes in residue classes (this problem was already stated in Gelfond's paper~\cite{Gelfond:68}). In their proofs, they used sophisticated exponential sum methods. Due to their paper, the area gained new impact and many new results have been proved in the past few years; see e.g.\ \cite{DMR:09,DMR:11,MR:09,Mor:10}.

The present paper is devoted to digit patterns and the sum-of-digits function for number systems with a rational number as base. 
We begin with the definition of the representation of positive integers that was given by Akiyama, Frougny, and Sakarovitch \cite{AFS:08}.
For given coprime integers $a, b$ with $a > b \ge 1$, let $a/b$ be the \emph{base} and $\mathcal{D} = \{0, 1, \ldots, a-1\}$ the set of \emph{digits}.
Then every positive integer $n$ has a unique finite representation of the form
\begin{equation}\label{rationalrepresentation}
n = \frac{1}{b} \sum_{k= 0}^{\ell(n)-1} \varepsilon_k(n)\, \left(\frac{a}{b}\right)^k, \quad \varepsilon_k(n) \in \mathcal{D},
\end{equation}
with $\varepsilon_{\ell(n)-1}(n) \ne 0$. We call the pair $(a/b, \{0,1,\ldots,a-1\})$ a \emph{rational base number system}. The representation in \eqref{rationalrepresentation} is the \emph{representation of~$n$ in base~$a/b$}. For example, the representations of $1, 2, \ldots, 10$ in base~$3/2$ are
\begin{align*}
(2)&=1,\quad &(21)&=2, \quad &(210)&=3, \quad &(212)&=4, \quad &(2101)&=5,\\
(2120)&=6,\quad &(2122)&=7,\quad &(21011)&=8,\quad &(21200)&=9,\quad &(21202)&=10.
\end{align*}
Note that --- when $b > 1$ holds --- these representations are different from the $\beta$-expansions with $\beta = a/b$ that were defined by R\'enyi~\cite{Renyi57}. 

One of the motivations to study these number systems is their relation to Mahler's $\frac32$-problem that was pointed out in \cite[Section~6]{AFS:08}. Mahler asked whether there exists $z\in\mathbb{R}\setminus\{0\}$ such that the fractional part of $z(3/2)^n$ falls into $[0, 1/2)$ for all $n\ge0$; see \cite{Mahler:68}. Among other things, Akiyama, Frougny, and Sakarovitch  could prove with the help of the base $3/2$ number system that there exist infinitely many $z\in\mathbb{R}\setminus\{0\}$ such that $\| z (3/2)^n\|<1/3$ holds for all $n\ge0$, where $\|x\|$ denotes the distance of $x$ to its nearest integer; see~\cite[Corollary 4]{AFS:08}.
These number systems also have connections to the Josephus problem; see \cite[Section~4.4]{AFS:08} and \cite[Example~2.1]{ST:11}. 

In \cite{AFS:08}, the authors put some emphasis on the investigation of the language~$L$ of words on the alphabet~$\mathcal{D}$ defined by the representations of the positive integers in base~$a/b$.
It turns out that $L$ is not regular and even not context-free (in the case $b>1$); see~\cite[Corollaries 7 and~9]{AFS:08}. 
This makes it hard to get distribution results for the patterns of their digit strings as well as their sum-of-digits function. 
The present paper contains first results in this direction. 
Using Fourier analysis in the ad\`ele ring $\mathbb{A}_\mathbb{Q}$ of $\mathbb{Q}$ as well as character sum estimates, we are able to prove that each digit string of a given length~$\ell$ occurs in the representations of the integers $1,2, \ldots, N$ in base $a/b$ with equal frequency~$a^{-\ell}$. We also provide an error term. As a corollary, we state a result in the spirit of
Delange~\cite{D75} on the summatory function of the sum-of-digits function. Moreover, we give a new construction of normal numbers defined in terms of representations in base $a/b$.

To be able to state our main result, we define the function $S_w(N)$, which counts the number of occurrences of the pattern $w$ in the  base $a/b$ representations
of the first $N$ positive integers. In particular, let $w = (w_{\ell-1},\ldots,w_1,w_0)$, with $w_i\in\mathcal{D}$, be a finite sequence of digits. The length of $w$ is denoted by $|w|=\ell$. We set
\[
S_w(N) = \sum_{k=0}^{\ell(N)-|w|} S_{k,w}(N)
\]
with
\[
S_{k,w}(N) = \#\{ 1\le n\le N:\, \ell(n) \ge k + |w|,\, (\varepsilon_{k+|w|-1}(n), \ldots, \varepsilon_{k+1}(n), \varepsilon_{k}(n))=w\}.
\]

\begin{theorem} 
Let  $a, b \in \mathbb{N}$  with  $(a,b)=1$ and  $a > b \ge 1$. Then
\[
S_w(N) = \frac{N}{a^{|w|}} \log_{a/b} N + \mathcal{O}\big(N \log\log N \big).
\]
\end{theorem}

The arithmetic function $s_{a/b}(n) = \sum_{0 \le k < \ell(n)} \varepsilon_k(n)$
is well-defined and is called the \emph{sum-of-digits function in base~$a/b$}.
From our main theorem we obtain the following corollary on the summatory function of~$s_{a/b}$.

\begin{corollary}\label{c1}
Let $a, b \in \mathbb{N}$ with $(a,b)=1$ and  $a > b \ge 1$. Then
\[
\sum_{n=1}^N s_{a/b}(n) = \frac{a-1}{2} N \log_{a/b} N + \mathcal{O}\big(N \log\log N \big).
\]
\end{corollary}

The next definition generalizes the famous Champernowne constant~$\mathfrak{c}$, the decimal expansion of which is given by $\mathfrak{c}=0\, \decdot\, 1\,2\,3\,4\,5\,6\,7\,8\,9\,10\,11\,12\,\cdots.$
Let $\mathfrak{z}_{a/b}$ be the unique real number in $(0,1)$ that has a (standard) representation in base $a$  which is obtained by concatenating the digits of $1,2,3,\ldots$ in their representation in base~$a/b$, that is
\[
\mathfrak{z}_{a/b} = 0 \, \decdot \, \varepsilon_0(1)\, \varepsilon_{\ell(2)-1}(2) \cdots \varepsilon_{1}(2) \varepsilon_{0}(2) \, \varepsilon_{\ell(3)-1}(3) \cdots \varepsilon_{1}(3) \varepsilon_{0}(3) \, \varepsilon_{\ell(4)-1}(4) \cdots \varepsilon_{1}(4) \varepsilon_{0}(4)\,\cdots
\]
or, more precisely,
\[
\mathfrak{z}_{a/b} = \sum_{n=1}^\infty \sum_{k=1}^{\ell(n)} \frac{\varepsilon_{\ell(n)-k}(n)}{a^{k+\sum_{j=1}^{n-1}\ell(j)}}\,.
\]
For example, the real number $\mathfrak{z}_{3/2}$ is given in base $3$ by
\[
\mathfrak{z}_{3/2} = 0\,\decdot\,2\,21\,210\,212\,2101\,2120\,2122\,21011\,\cdots.
\]
In the case that $a=10$ and $b=1$, $\mathfrak{z}_{a/b}$ is exactly the constant~$\mathfrak{c}$. Champernowne~\cite{Champernowne:33} proved that $\mathfrak{c}$ is normal in base~$10$, i.e., every digit string $w$ of length $|w|$ occurs with equal frequency $10^{-|w|}$ in the base $10$ representation of $\mathfrak{c}$. We have the following result.

\begin{corollary}\label{c2}
Let $a, b \in \mathbb{N}$ with $(a,b)=1$ and  $a > b\ge 1$. Then $\mathfrak{z}_{a/b}$ is normal in base~$a$.
\end{corollary}

\begin{remark}
We mention here that all our results and proofs remain valid (with obvious changes) if the basis $a/b$ is negative and the set of digits is $\mathcal{D}=\{0,1,\ldots,|a|-1\}$. In order to keep the exposition as simple as possible we confine ourselves to positive integers $a$ and $b$ throughout this paper.
\end{remark}

\section{Plan of the proof}\label{section:plan}
To analyze~$S_w(N)$, it is important to understand the structure of the digits of $n \le N$ with respect to the given base. In the case of the standard $q$-ary representation in~$\mathbb{N}$ (i.e., $a = q$, $b = 1$), this is relatively simple. Indeed,
we have $n/q^{k+1} = \sum_{0 \le j < \ell(n)} \varepsilon_j(n)\, q^{j-k-1}$, and this sum can be split into
\begin{equation}\label{digitformula}
\sum_{k < j < \ell(n)} \varepsilon_j(n)\, q^{j-k-1} \in \mathbb{Z}[q] = \mathbb{Z} \quad \mbox{and} \quad \sum_{0 \le j \le k} \varepsilon_j(n)\, q^{j-k-1} \in \bigg[\frac{\varepsilon_k(n)}{q}, \frac{\varepsilon_k(n)+1}{q}\bigg).
\end{equation}
Therefore, the digit $\varepsilon_k(n)$ in base~$q$ is equal to~$d\in\{0, \ldots, q-1\}$ if and only if the fractional part of $n/q^{k+1}$ lies in the half-open interval $[d/q, (d+1)/q)$.
This makes it easy to employ analytic methods in order to study the frequency of certain digit patterns. The first problem consists in finding related formulas for rational base number systems.
However, the most crucial difference to the standard base~$q$ representations is the already discussed fact that the language~$L$ of words on the alphabet $\mathcal{D}$ defined by the representations of the positive integers in base~$a/b$ cannot be described in a simple way. Thus, there seem to be no elementary combinatorial considerations proving our results.

To overcome the difficulties occurring in the case of rational base number systems we have to embed the rational numbers into a subring $\mathbb{K}_{a/b}$ of the ad\`ele ring $\mathbb{A}_\mathbb{Q}$, where the (embedding of the) set $\mathbb{Z}[a/b] = \mathbb{Z}[1/b]$ forms a lattice; see Section~\ref{section:fourier} for a precise definition.
This makes it possible to find a compact, self-affine \emph{fundamental domain} $\mathcal{F}$ of $\mathbb{Z}[1/b]$ in $\mathbb{K}_{a/b}$ which is related to the underlying rational number system. Using its self-affine structure, the set $\mathcal{F}$ can be written as a union of subsets~$\mathcal{F}_d$ ($d\in\mathcal{D}$), which play a similar role as the intervals $[d/q,(d+1)/q)$ do for $q$-ary representations in \eqref{digitformula}; see Figure~\ref{fig:Fd32} for a picture of these sets. With help of these sets we are able to establish formulas similar to the one in \eqref{digitformula} in the context of rational base number systems. This enables us to use Fourier analysis in $\mathbb{K}_{a/b}$ in order to reformulate the problem of counting the digit patterns as a problem on estimating character sums. Since $\mathbb{K}_{a/b}$ is a subring of the ad\`ele ring~$\mathbb{A}_\mathbb{Q}$, the characters of this ad\`ele ring are used to define the appropriate Fourier transformation and its inverse; see e.g.\ Tate's thesis~\cite{Tate:67} or Weil~\cite{Weil:73}. Finally, using ideas reminiscent of Drmota et al.\!~\cite{DMR:09}, we estimate the character sums emerging from this Fourier transformation process.

The paper is organized as follows. In Section~\ref{section:fourier}, we give some definitions and set up the environment for the required Fourier analysis. This includes analyzing the Pontryagin dual of the fundamental domain~$\mathcal{F}$ as well  as the calculation of Fourier coefficients  for Urysohn functions related to~$\mathcal{F}$. In Section~\ref{section:proofthm}, we perform the estimates of the character sums emerging from the Fourier transformation process in order to prove the main results.
Finally, in Section~\ref{section:perspectives} we briefly discuss some open questions and possible directions of future research related to the topic of the present paper.

\section{Fourier analysis on fundamental domains}\label{section:fourier}

Throughout this paper, we set
\[
\alpha = a/b.
\]
In this section, we study self-affine tiles associated with representations in rational bases.
These tiles are special cases of the \emph{rational self-affine tiles} defined in~\cite{ST:11}. We need several notations and definitions.

For each (finite) rational prime $p$ denote the $p$-adic completion of~$\mathbb{Q}$ by  $\mathbb{Q}_p$, that is,
the field $\mathbb{Q}_p$ is the completion of $\mathbb{Q}$ w.r.t.\ the  topology induced by the $p$-adic exponent $v_p$. As usual, we write $\mathbb{Z}_p$ for the ring of integers of $\mathbb{Q}_p$. The fractional part of $x\in \mathbb{Q}_p$ will be denoted by $\lambda_p(x)$, i.e., $\lambda_p(\sum_{j=k}^\infty d_jp^j)=\sum_{j=k}^{-1} d_jp^j$ for all sequences $(d_j)_{j\ge k}$ with $d_j\in\{0,\ldots,p-1\}$, $k<0$. For the (unique) infinite prime $p=\infty$ set $\mathbb{Q}_p=\mathbb{R}$, which is the completion of~$\mathbb{Q}$ with respect to the  Archimedean absolute value $|\cdot|$ in~$\mathbb{Q}$.

Consider the set of primes
\[
S_\alpha = \{p:\, p\mid b,\, \mbox{$p$ prime}\} \cup \{\infty \}.
\]
With help of~$S_\alpha$ we define the subring
\[
\mathbb{K}_\alpha = \prod_{p \in S_\alpha} \mathbb{Q}_p
\]
of the ad\`ele ring~$\mathbb{A}_\mathbb{Q}$ of $\mathbb{Q}$.
We equip $\mathbb{K}_\alpha$ with the measure $\mu_\alpha$ that is defined by the product measure of the Lebesgue measure $\mu_\infty$ on $\mathbb{R}$ and the Haar measures\footnote{We choose $\mu_p$ in a way that $\mu_p(\mathbb{Z}_p)=1$. In particular, this implies $\mu_p(p^\ell \mathbb{Z}_p)= p^{-\ell}.$} $\mu_p$ on $\mathbb{Q}_p$ for $p \in S_\alpha\setminus\{\infty\}$.
Let
\[
\Phi:\ \mathbb{Q} \to \mathbb{K}_\alpha, \quad z \mapsto (z, \ldots, z),
\]
be the diagonal embedding of $\mathbb{Q}$ in $\mathbb{K}_\alpha$. The set $\mathbb{Q}$ acts multiplicatively on $\mathbb{K}_\alpha$ by
\[
\xi \cdot (z_p)_{p\in S_\alpha} = (\xi z_p)_{p\in S_\alpha}.
\]
Following \cite{ST:11}, we define $\mathcal{F} = \mathcal{F}(\alpha, \mathcal{D})$ as the unique non-empty compact subset of $\mathbb{K}_\alpha$ satisfying
\[
\mathcal{F} = \bigcup_{d\in\mathcal{D}} \alpha^{-1} \cdot \big(\mathcal{F} + \Phi(d)\big).
\]
This set can be written explicitly as
\[
\mathcal{F} = \bigg\{\sum_{k=1}^\infty \Phi(\varepsilon_k \alpha^{-k}):\, \varepsilon_k \in \mathcal{D}\bigg\}.
\]
In order to keep track of the occurrences of a given digit $d \in \mathcal{D}$ in a representation in base $a/b$, we will have to deal with the subsets
\[
\mathcal{F}_d = \alpha^{-1} \cdot \big(\mathcal{F} + \Phi(d)\big)\qquad(d \in\mathcal{D})
\]
of $\mathcal{F}$. See Figure~\ref{fig:Fd32} for a representation of these sets for the case $\alpha = 3/2$.

\begin{figure}[ht]
\includegraphics{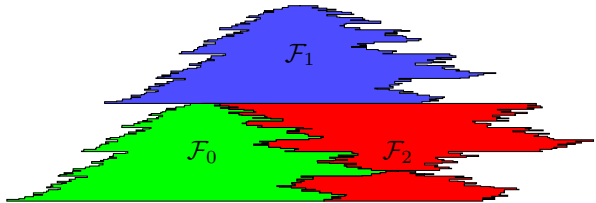}
\caption{The sets $\mathcal{F}_d$ constituting $\mathcal{F}  \in \mathbb{R} \times \mathbb{Q}_2$ for $\alpha = 3/2$.
Here, an element $\sum_{j=k}^\infty d_j \alpha^{-j}$ of $\mathbb{Q}_2$, with $d_j \in \{0,1\}$, is represented by $\sum_{j=k}^\infty d_j 2^{-j}$.}
\label{fig:Fd32}
\end{figure}

We also need to consider approximations of~$\mathcal{F}_d$. To this matter, define the ``boxes''
\[
D_r = \alpha^{-r} \cdot \bigg([0, 1] \times  \prod_{p\mid b} \mathbb{Z}_p \bigg) \qquad(r\ge 0).
\]
With help of these boxes, we define the approximations
\[
\mathcal{F}_{d,r} = \bigcup_{\varepsilon_2, \ldots, \varepsilon_r \in \mathcal{D}} \bigg(\Phi(d \alpha^{-1}) + \sum_{k=2}^r \Phi(\varepsilon_k \alpha^{-k}) + D_r\bigg) \qquad(r\ge 1)
\]
of~$\mathcal{F}_d$. (Note that the sum over $k$ is empty for $r=1$.)

Recall that a collection~$\mathcal{C}$ of compact subsets of~$\mathbb{K}_\alpha$ is a \emph{tiling} of~$\mathbb{K}_\alpha$ if each element of~$\mathcal{C}$ is the closure of its interior and if $\mu_\alpha$-almost every point of~$\mathbb{K}_\alpha$ is contained in exactly one element of~$\mathcal{C}$. It will be of importance later that $\mathcal{F}$, $\mathcal{F}_d$, and their approximations induce periodic tilings of~$\mathbb{K}_\alpha$. This is made precise in the following lemma.

\begin{lemma} \label{l:tiling}
The following collections are tilings of~$\mathbb{K}_\alpha$.
\smallskip
\begin{itemize}
\item
$\{\Phi(x) + D_r:\, x \in \alpha^{-r} \mathbb{Z}[\alpha]\} \quad (r\ge 0)$,
\smallskip
\item
$\{\Phi(x) + \mathcal{F}_{d,r}:\, x \in \mathbb{Z}[\alpha],\, d \in \mathcal{D}\} \quad (r\ge 1)$,
\smallskip
\item
$\{\Phi(x) + \mathcal{F}:\, x \in \mathbb{Z}[\alpha]\}$,
\smallskip
\item
$\{\Phi(x) + \mathcal{F}_d:\, x \in \mathbb{Z}[\alpha],\, d \in \mathcal{D}\}$.
\end{itemize}
\end{lemma}

\begin{proof}
To show that $\{\Phi(x) + D_0:\, x \in \mathbb{Z}[\alpha]\}$ is a tiling of~$\mathbb{K}_\alpha$, consider a point $(z_p)_{p\in S_\alpha} \in \mathbb{K}_\alpha$, and let
\[
y = \sum_{p\in S_\alpha\setminus\{\infty\}} \lambda_p(z_p) + \Bigg\lfloor z_\infty - \sum_{p\in S_\alpha\setminus\{\infty\}} \lambda_p(z_p)\Bigg\rfloor.
\]
Since $\lambda_p(z_p)\in\mathbb{Z}[\alpha]$ holds for each $p\in S_\alpha\setminus\{\infty\}$, we have $y\in \mathbb{Z}[\alpha]$.
By the definition of~$y$, we conclude that $(z_p)_{p\in S_\alpha} \in \Phi(y) + D_0$, and $(z_p)_{p\in S_\alpha} \not\in \Phi(x) + D_0$ for all $x \in \mathbb{Z}[\alpha] \setminus \{y\}$ except when $z_\infty - \sum_{p\in S_\alpha\setminus\{\infty\}} \lambda_p(z_p) \in \mathbb{Z}$.
As $\sum_{p\in S_\alpha\setminus\{\infty\}} \lambda_p(z_p) \in \mathbb{Q}$, the set of points $(z_p)_{p\in S_\alpha} \in \mathbb{K}_\alpha$ with $z_\infty - \sum_{p\in S_\alpha\setminus\{\infty\}} \lambda_p(z_p) \in \mathbb{Z}$ has $\mu_\alpha$-measure zero (note that $z_\infty$ can be an arbitrary real number).
Hence, $\{\Phi(x) + D_0:\, x \in \mathbb{Z}[\alpha]\}$ is a tiling of~$\mathbb{K}_\alpha$.
Multiplying by $\alpha^{-r}$ yields that $\{\Phi(x) + D_r:\, x \in \alpha^{-r} \mathbb{Z}[\alpha]\}$ is a tiling of~$\mathbb{K}_\alpha$ for each $r\ge 0$.
Since the set $\{\varepsilon_1 \alpha^{-1} + \cdots + \varepsilon_r \alpha^{-r}:\, \varepsilon_1, \ldots, \varepsilon_r \in \mathcal{D}\}$ forms a complete residue system of $\mathbb{Z}[\alpha] / \alpha^{-r} \mathbb{Z}[\alpha]$, it follows that $\{\Phi(x) + \mathcal{F}_{d,r}:\, x \in \mathbb{Z}[\alpha],\, d \in \mathcal{D}\}$ is a tiling of~$\mathbb{K}_\alpha$ for each $r\ge 1$.

By \cite[Theorem~2]{ST:11}, the collection $\{\Phi(x) + \mathcal{F}:\, x \in \mathbb{Z}[\alpha]\}$ forms a tiling of~$\mathbb{K}_\alpha$.
Now, multiplying by $\alpha^{-1}$ yields that $\{\Phi(x) + \mathcal{F}_d:\, x \in \mathbb{Z}[\alpha],\, d \in \mathcal{D}\}$ is a tiling of~$\mathbb{K}_\alpha$.
\end{proof}

A~patch of the tiling $\{\Phi(x) + \mathcal{F}:\, x \in \mathbb{Z}[\alpha]\}$ for $\alpha = 3/2$ is depicted in Figure~\ref{fig:F32}.

\begin{figure}[ht]
\includegraphics[width=14cm]{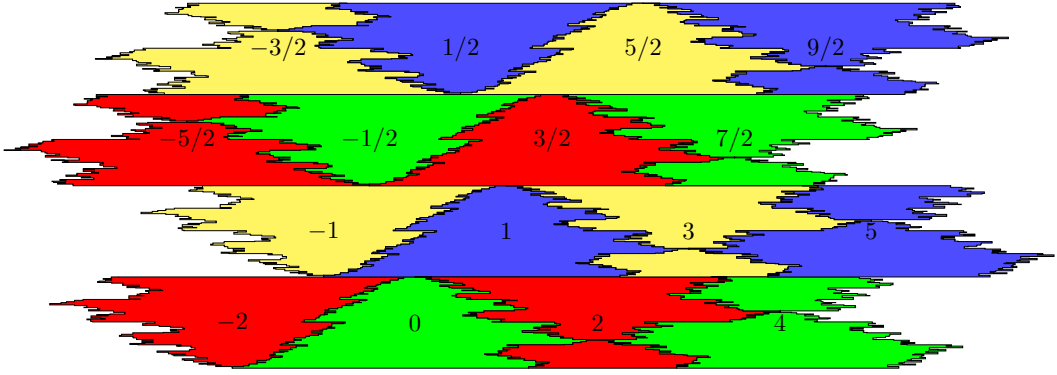}
\caption{The tiles $\Phi(x) + \mathcal{F} \in \mathbb{R} \times \mathbb{Q}_2$ for $\alpha = 3/2$, $x \in \{\frac{-5}{2},\frac{-4}{2},\ldots,\frac{10}{2}\}$.}
\label{fig:F32}
\end{figure}

For a set $M \subset \mathbb{K}_\alpha$, we denote the characteristic function of~$M$ by~$\mathbf{1}_M$. Since
\[
 \int_{\mathbb{K}_\alpha} \mathbf{1}_{D_r}\, \mathrm{d} \mu_\alpha = \alpha^{-r} \prod_{p \mid b} \mu_p(p^{v_p(b)r}\mathbb{Z}_p) = \alpha^{-r} \prod_{p \mid b}  p^{-v_p(b)r} =\alpha^{-r}b^{-r}= a^{-r},
\]
the function
\[
f_{d,r}(\mathbf{z}) = a^r \int_{D_r} \mathbf{1}_{\Phi(\mathbb{Z}[\alpha])+\mathcal{F}_{d,r}}(\mathbf{z} + \mathbf{y})\, \mathrm{d} \mu_\alpha(\mathbf{y})
\]
is an Urysohn function which approximates the characteristic function of $\mathcal{F}_d \bmod \Phi(\mathbb{Z}[\alpha])$.
It can be split up into a sum of the $a^{r-1}$ functions
\[
g_{x,r}(\mathbf{z}) = a^r \int_{D_r} \mathbf{1}_{\Phi(x+\mathbb{Z}[\alpha])+D_r}(\mathbf{z}+\mathbf{y})\, \mathrm{d}\mu_\alpha(\mathbf{y})
\]
with $x \in \{d\alpha^{-1} + \varepsilon_2 \alpha^{-2} + \cdots + \varepsilon_r \alpha^{-r}:\, \varepsilon_2,\ldots, \varepsilon_r \in\mathcal{D}\}$.

We want to expand $f_{d,r}$ into a Fourier series. Since $\{\Phi(x)+D_0: x\in \mathbb{Z}[\alpha]\}$ forms a tiling of~$\mathbb{K}_\alpha$ by Lemma~\ref{l:tiling}, and since $f_{d,r}$ and $g_{x,r}$ are periodic functions mod $\Phi(\mathbb{Z}[\alpha])$, we will do Fourier analysis on the compact fundamental domain $D_0$ of $\mathbb{K}_\alpha / \Phi(\mathbb{Z}[\alpha])$. To this matter, we set $\e(y)=\exp(2\pi i y)$ and define
the character
\[
\chi\big((z_p)_{p\in S_\alpha}\big) = \e\!\Bigg(\sum_{p\in S_\alpha\setminus\{\infty\}}\lambda_p(z_p) - z_\infty\Bigg).
\]
(For the definition of characters of the ad\`ele ring~$\mathbb{A}_\mathbb{Q}$, we refer to Weil \cite[Section~IV.2, p.~66]{Weil:73}.)
We will also use the notation $\tilde\chi(\xi) = \chi(\Phi(\xi))$. We have the following property.

\begin{lemma}\label{lemCharacter}
Let $\xi\in\mathbb{Q}$. Then $\tilde\chi(\xi)=1$ if and only if $\xi \in \mathbb{Z}[\alpha]$.
\end{lemma}

\begin{proof}
 The definition of the character $\chi$ implies that $\tilde\chi(\xi)=1$ if and only if
\begin{align}\label{eq:character}
 \sum_{p\mid b} \lambda_p(\xi)-\xi \in \mathbb Z.
\end{align}
Assume now that $\xi\not\in\mathbb Z[\alpha]$. Then there exists a prime $q\nmid b$ such that $\xi\not\in\mathbb Z_q$ (to be precise, the canonical embedding of $\xi$ in~$\mathbb{Q}_q$ is not an element of~$\mathbb{Z}_q$). Since $\lambda_p(\xi)\in \mathbb Z_q$ for all $p\mid b$, we see that $\sum_{p\mid b} \lambda_p(\xi) - \xi \not\in\mathbb Z_q$. Hence,~\eqref{eq:character} cannot hold true. It remains to show that $\xi\in\mathbb Z[\alpha]$ implies~\eqref{eq:character}. Let $q$ be prime. If $q\nmid b$ then $\lambda_p(\xi)\in\mathbb Z_q$ for all $p\mid b$ and $\xi\in\mathbb Z_q$. Thus,  $\sum_{p\mid b} \lambda_p(\xi) - \xi \in\mathbb Z_q$. If $q\mid b$, then $\lambda_p(\xi)\in \mathbb Z_q$ for $p \ne q$ and $\lambda_q(\xi) - \xi \in \mathbb Z_q$. Hence,  $\sum_{p\mid b} \lambda_p(\xi) - \xi \in\mathbb Z_q$ in this case too. Thus  $\sum_{p\mid b} \lambda_p(\xi)-\xi \in \bigcap_q \mathbb{Z}_q=\mathbb{Z}$, which proves the lemma.
\end{proof}

Using \cite[Theorem~23.25]{Hewitt-Ross:63}, we see that the
Pontryagin dual $\widehat{D_0}$ of $D_0$ is isomorphic to the annihilator of $ \Phi(\mathbb{Z}[\alpha])$ in the Pontryagin dual $\widehat{\mathbb{K}_\alpha}$. Therefore, using \cite[Lemma~4.7]{ST:11},
we conclude that $\widehat{D_0} = \{\chi(\xi\,\cdot)\;:\; \xi \in \mathbb{Z}[\alpha]^*\}$ holds with
\[
\mathbb{Z}[\alpha]^* = \{\xi \in \mathbb{Q}:\, \tilde\chi(\xi x) =1 \hbox{ for all } x \in \mathbb{Z}[\alpha]\}.
\]
Lemma~\ref{lemCharacter} implies that $\mathbb{Z}[\alpha]^* = \mathbb Z[\alpha]$; compare with \cite[proof of Lemma~4.1.5]{Tate:67} and \cite[Lemma~4.6]{ST:11}. Summing up, we proved the following lemma.

\begin{lemma}\label{lemPontryagin}
The Pontryagin dual of $D_0$ is given by
\[
\widehat{D_0} = \{\chi(\xi\,\cdot)\;:\; \xi \in \mathbb{Z}[\alpha]\}.
\]
\end{lemma}

By Lemma~\ref{lemPontryagin}, the Fourier expansions of the functions $g_{x,r}$ and $f_{d,r}$ are of the form
\begin{equation}\label{fourierseries}
g_{x,r}(\mathbf{z}) \sim \sum_{\xi \in \mathbb{Z}[\alpha]} c_{x,r,\xi}\, \chi(\xi \cdot \mathbf{z})
\quad\hbox{and}\quad
f_{d,r}(\mathbf{z}) \sim \sum_{\xi\in\mathbb{Z}[\alpha]} c'_{d,r,\xi}\, \chi(\xi \cdot \mathbf{z}).
\end{equation}
We will show that these Fourier expansions converge pointwise to the respective function. To this matter we need the following estimates of the Fourier coefficients of $g_{x,r}$.

\begin{lemma} \label{lem3}
For $r\ge 0$ and $x \in \{d\alpha^{-1} + \varepsilon_2 \alpha^{-2} + \cdots + \varepsilon_r \alpha^{-r}:\, d,\varepsilon_2,\ldots, \varepsilon_r \in\mathcal{D}\}$ we have
\[
c_{x,r,\xi} = \begin{cases}a^{-r} & \mbox{if}\ \xi = 0, \\ \alpha^r b^{-r}\, \tilde\chi(-x \xi)\,  \frac{|1-\e(\alpha^{-r}\xi)|^2}{4\,\xi^2\pi^2} & \mbox{if}\ \xi \in \frac{\mathbb{Z}}{b^r} \setminus \{0\}, \\ 0 & \mbox{otherwise.}\end{cases}
\]
\end{lemma}

\begin{proof}
As $\Phi(x) + D_0$ is a fundamental domain of $\mathbb{K}_\alpha / \Phi(\mathbb{Z}[\alpha])$, Lemma~\ref{lemPontryagin} implies that the Fourier coefficient $c_{x,r,\xi}$ of $g_{x,r}(\mathbf{z})$ is given by
\begin{align*}
c_{x,r,\xi} & =  a^r\int_{\Phi(x)+D_0} g_{x,r}(\mathbf{z})\, \chi(-\xi \cdot \mathbf{z})\, \mathrm{d}\mu_\alpha(\mathbf{z}) \\
& = a^r \int_{\Phi(x)+D_0} \mu_\alpha\Big(\big(\Phi(x+\mathbb{Z}[\alpha]) + D_r\big) \cap \big(\mathbf{z} + D_r\big)\Big)\, \chi(-\xi \cdot \mathbf{z})\, \mathrm{d}\mu_\alpha(\mathbf{z}) \\
& = a^r\, \chi\big(-\xi \cdot \Phi(x)\big) \int_{D_0} \mu_\alpha\Big(\big(\Phi(\mathbb{Z}[\alpha]) + D_r\big) \cap \big(\mathbf{z} + D_r\big)\Big)\, \chi(-\xi \cdot \mathbf{z})\, \mathrm{d}\mu_\alpha(\mathbf{z}).
\end{align*}
Since $D_r \subseteq D_0$ and $D_0 + D_r \subseteq  D_0 \cup (\Phi(1)+D_0)$, we have
\[
\mu_\alpha\Big(\big(\Phi(\mathbb{Z}[\alpha]) + D_r\big) \cap \big(\mathbf{z} + D_r\big)\Big) = \mu_\alpha\Big(D_r \cap \big(\mathbf{z} + D_r\big)\Big) + \mu_\alpha\Big(\big(\Phi(1)+D_r \big) \cap \big(\mathbf{z} + D_r\big)\Big)
\]
for $\mathbf{z} \in D_0$.
Using the identities $\chi(-\xi \cdot (\mathbf{z} - \Phi(1))) = \chi(-\xi \cdot \mathbf{z})$ and $D_0 \cup (-\Phi(1)+D_0) = [-1,1] \times \prod_{p\mid b}\mathbb{Z}_p$, we obtain that
\[
c_{x,r,\xi} = a^r\, \tilde\chi(-x \xi) \int_{[-1,1] \times  \prod_{p \mid b} \mathbb{Z}_p} \mu_\alpha\big(D_r \cap (\mathbf{z} + D_r)\big)\, \chi(-\xi \cdot \mathbf{z})\, \mathrm{d}\mu_\alpha(\mathbf{z}).
\]
Setting
\begin{align*}
I_\infty & = \int_{-1}^1 \mu_\infty\big([0, \alpha^{-r}] \cap [z_\infty, z_\infty + \alpha^{-r}]\big) \e(\xi z_\infty)\, \mathrm{d} z_\infty\quad\hbox{and} \\
I_p & = \int_{\mathbb{Z}_p} \mu_p\big(p^{rv_p(b)} \mathbb{Z}_p \cap (z_p + p^{rv_p(b)} \mathbb{Z}_p)\big) \e\!\big(-\lambda_p(\xi z_p)\big)\, \mathrm{d} \mu_p(z_p) \qquad (p\in S_\alpha\setminus\{\infty\}),
\end{align*}
Fubini's theorem implies that
$c_{x,r,\xi} = a^r \tilde\chi(-x\xi)\prod_{p \in S_\alpha} I_p$.
We have
\[
I_\infty = \int_{-\alpha^{-r}}^{\alpha^{-r}} (\alpha^{-r} - | z_\infty|) \e(\xi z_\infty)\, \mathrm{d}z_\infty = \begin{cases}\alpha^{-2r} & \mbox{if}\ \xi = 0, \\ \frac{(1-\e(\alpha^{-r} \xi))\, (1 - \e(-\alpha^{-r} \xi))}{4\,\xi^2\pi^2} & \mbox{otherwise.}\end{cases}
\]
It remains to calculate the integrals $I_p$ for $p\mid b$. Note that
\[
 I_p = \mu_p\big(p^{r v_p(b)}\mathbb{Z}_p\big) \int_{p^{r v_p(b)} \mathbb{Z}_p} \e\!\big( - \lambda_p(\xi z_p)\big)\, \mathrm{d}\mu_p(z_p).
\]
If $v_p(\xi) \ge -rv_p(b)$, then we get
\[
\int_{p^{rv_p(b)}\mathbb{Z}_p} \e\!\big(-\lambda_p(\xi z_p)\big)\, \mathrm{d}\mu_p(z_p) =  \int_{p^{rv_p(b)}\mathbb{Z}_p} \mathrm{d}\mu_p = p^{-rv_p(b)}.
\]
If $v_p(\xi) < -rv_p(b)$, then $\ell=-v_p(\xi) -rv_p(b)$ is a positive integer and we obtain
\begin{align*}
 \int_{p^{rv_p(b)}\mathbb{Z}_p} \e\!\big(-\lambda_p(\xi z_p)\big)\, \mathrm{d}\mu_p(z_p) &= p^{v_p(\xi)}  \int_{p^{-\ell}\mathbb{Z}_p} \e\!\big(-\lambda_p(z_p)\big)\, \mathrm{d}\mu_p(z_p)\\
&= p^{v_p(\xi)} \sum_{k=0}^{p^\ell-1} \e\!\left(-\frac{k}{p^\ell}\right) = 0.
\end{align*}
Thus we have for $p\mid b$ that
\[
I_p  = \begin{cases}p^{-2rv_p(b)} & \mbox{if}\ v_p(\xi) \ge -rv_p(b), \\ 0 & \mbox{otherwise.}\end{cases}
\]
Putting everything together proves the desired result.
\end{proof}

We now state the convergence result for the Fourier series in \eqref{fourierseries}.

\begin{lemma}\label{pointwise}
For each $\mathbf{z} \in \mathbb{K}_\alpha$ we have
\begin{equation*}
g_{x,r}(\mathbf{z}) = \sum_{\xi \in \mathbb{Z}[\alpha]} c_{x,r,\xi}\, \chi(\xi \cdot \mathbf{z})
\quad\hbox{and}\quad
f_{d,r}(\mathbf{z}) = \sum_{\xi\in\mathbb{Z}[\alpha]} c'_{d,r,\xi}\, \chi(\xi \cdot \mathbf{z}),
\end{equation*}
i.e., the Fourier series of  $g_{x,r}$ and $f_{d,r}$ converge pointwise.
\end{lemma}

\begin{proof}
We just show the assertion for $g_{x,r}$, the one for $f_{d,r}$ then follows immediately as
\[
f_{d,r}(\mathbf{z})= \sum_{x \in \{d\alpha^{-1} + \varepsilon_2 \alpha^{-2} + \cdots + \varepsilon_r \alpha^{-r}:\, \varepsilon_2,\ldots, \varepsilon_r \in\mathcal{D}\}} g_{x,r}(\mathbf{z}).
\]

Since $g_{x,r}$ is periodic, we can regard it as a function on~$D_0$. In particular, $g_{x,r}\in L^2(D_0)$. Thus Plancherel's theorem (see e.g.\ \cite[Theorem~31.18]{Hewitt-Ross:70}) implies
that $g_{x,r}$ is equal to its Fourier series $\mu_\alpha$-almost everywhere. Since $g_{x,r}$ is
continuous, the lemma is established if we show that its Fourier series converges to a continuous function. Lemma~\ref{lem3} implies that
\begin{equation}\label{ff1}
\sum_{\xi \in \mathbb{Z}[\alpha]} c_{x,r,\xi}\, \chi(\xi \cdot \mathbf{z}) = \sum_{\xi \in b^{-r}\mathbb{Z}} c_{x,r,\xi}\, \chi(\xi \cdot \mathbf{z}).
\end{equation}
Moreover, by the same lemma, for each $\varepsilon > 0$ there exists $N\in \mathbb{N}$ such that
\begin{equation}\label{ff2}
\sum_{\begin{subarray}{c} \xi \in b^{-r}\mathbb{Z} \\ |\xi| > N \end{subarray}} |c_{x,r,\xi}| < \varepsilon.
\end{equation}
Since the partial sums $\sum_{\xi \in b^{-r}\mathbb{Z}, \, |\xi| \le N} c_{x,r,\xi}\, \chi(\xi \cdot \mathbf{z})$ are obviously continuous functions, the convergence of the series in \eqref{ff1} to a continuous function follows because \eqref{ff2} implies that it converges uniformly in $\mathbf{z}$.
\end{proof}

Later we will need the following estimate for the Fourier coefficients of $f_{d,r}$.

\begin{lemma}\label{lem6}
For $r\ge 0$ and $d\in\mathcal{D}$ we have
\[
c'_{d,r,\xi} = \begin{cases}a^{-1} & \mbox{if}\ \xi = 0, \\ \mathcal{O}(\min(1,\alpha^{2r} \xi^{-2})) & \mbox{if}\ \xi \in \frac{\mathbb{Z} \setminus a \mathbb{Z}}{b^r}, \\  0 & \mbox{otherwise.}\end{cases}
\]
\end{lemma}

\begin{proof}
As $c'_{d,r,\xi} = \sum_{\varepsilon_2,\ldots,\varepsilon_r\in\mathcal{D}} c_{d\alpha^{-1}+\varepsilon_2\alpha^{-2}+\cdots+\varepsilon_r\alpha^{-r},r,\xi}$,
Lemma~\ref{lem3} yields the stated formulas for $\xi = 0$ and $\xi \not\in \frac{\mathbb{Z}}{b^r}$.
Since $|1-\e(x)|\le\min(2, 2\pi |x|)$ for all $x\in\mathbb{R}$, the stated formula also holds in the case $\xi \in \frac{\mathbb{Z} \setminus a \mathbb{Z}}{b^r}$.
It remains to show that $c'_{d,r,\xi} = 0$ for $\xi \in \frac{a \mathbb{Z}}{b^r} \setminus \{0\}$.
By Lemma~\ref{lem3} we have
\begin{equation} \label{cstrich}
c'_{d,r,\xi} = \alpha^r b^{-r}\, \frac{|1-\e(\alpha^{-r}\xi)|^2}{4\,\xi^2\pi^2}\, \tilde\chi\bigg(-\frac{d \xi}{\alpha}\bigg) \sum_{\varepsilon_2,\ldots,\varepsilon_r\in\mathcal{D}} \tilde\chi\bigg(-\bigg(\frac{\varepsilon_2}{\alpha^2}+\cdots+\frac{\varepsilon_r}{\alpha^r}\bigg)\, \xi\bigg).
\end{equation}
Let $j \ge 1$ be maximal such that $a^j \mid b^r \xi$.
If $j \ge r$, then $c'_{d,r,\xi} = 0$ because $\e(\alpha^{-r}\xi) = 1$. Thus we may assume that $j<r$.
Since $-a \xi/\alpha^{j+1} \in \mathbb Z[\alpha]$ and $- \xi/\alpha^{j+1} \not\in \mathbb Z[\alpha]$, Lemma~\ref{lemCharacter} implies that $\tilde\chi(-a \xi/\alpha^{j+1}) = 1$ and $\tilde\chi(-\xi/\alpha^{j+1}) \ne 1$. This means that $\tilde\chi(-\xi/\alpha^{j+1})$ is a non-trivial $a$-th root of unity, thus
\[
\sum_{\varepsilon_{j+1}\in\mathcal{D}} \tilde\chi(-\varepsilon_{j+1} \xi/\alpha^{j+1}) = \sum_{\ell=0}^{a-1} \tilde\chi(-\xi/\alpha^{j+1})^\ell=0.
\]
Factoring the sum in \eqref{cstrich} accordingly, we obtain that $c'_{d,r,\xi} = 0$ in case $j < r$ as well.
\end{proof}

\section{Proof of the main results}\label{section:proofthm}

Using the Fourier expansions of $g_{x,r}$ and $f_{d,r}$ we will now reformulate our main theorem as an
estimation problem for character sums. We start with an easy result on the length of the representation in base $\alpha$.

\begin{lemma}\label{lem5}
Let $n>0$ be an integer whose representation in base $\alpha$ is of the form
\[
n = \frac{1}{b} \sum_{k=0}^{\ell(n)-1} \varepsilon_k(n) \alpha^k, \quad \varepsilon_k(n) \in \mathcal{D},
\]
with $\varepsilon_{\ell(n)-1}(n) \ne 0$.
Then we have
\[
\ell(n) = \log_\alpha(n) + \mathcal{O}(1).
\]
\end{lemma}

\begin{proof}
 We have
\[
\frac{1}{b} \alpha^{\ell(n)-1} \le n \le \frac{1}{b} \sum_{k=0}^{\ell(n)-1} (a-1)\, \alpha^k = \frac{a-1}{a-b} (\alpha^{\ell(n)}-1).
\]
Taking the logarithm (with respect to base~$\alpha$) implies the desired result.
\end{proof}

Set $\varepsilon_k(n) = 0$ for $k \ge \ell(n)$, and
\[
S'_{k,w}(N)  = \#\big\{1\le n\le N:\, \big(\varepsilon_{k+|w|-1}(n), \ldots, \varepsilon_{k}(n)\big) = w\big\}.
\]
By Lemma~\ref{lem5} we have  $\#\{n \in \mathbb{N} \, :\, \ell(n)<k+|w|\} \ll \alpha^k$. Hence, we get
\[
S_{k,w}(N) =  S'_{k,w}(N) + \mathcal{O}(\alpha^k).
\]
We trivially have $S_{k,w}(N) \le N$ for all $k \ge 0$.
Set $L = \ell(N) \approx \log_\alpha N$, and let $M \le L/2$ be a positive integer that we choose at the end of the proof.
Then we have
\begin{align}\label{eq:8}
S_w(N) = \sum_{0\le k< L} S_{k,w}(N) & = \sum_{M \le k \le L-M} S_{k,w}(N) + \mathcal{O}(N M) \notag \\
& = \sum_{M\le k\le L-M} S'_{k,w}(N) + \mathcal{O}(N M).
\end{align}

Since
\[
\frac{b n}{\alpha^{k+1}} = \sum_{k<j<\ell(n)} \varepsilon_j(n) \alpha^{j-k-1} + \varepsilon_k(n) \alpha^{-1} + \sum_{0\le j< k} \varepsilon_j(n) \alpha^{j-k-1},
\]
we have
\[
\Phi\Big(\frac{b n}{\alpha^{k+1}}\Big) \in \{\Phi(x) + \mathcal{F}_{\varepsilon_k(n)}:\, x \in \mathbb{Z}[\alpha]\},
\]
which is the appropriate analog of \eqref{digitformula} in our setting.
By Lemma~\ref{l:tiling}, $\{\Phi(x) + \mathcal{F}_d:\, x \in \mathbb{Z}[\alpha],\, d \in \mathcal{D}\}$ forms a tiling of~$\mathbb{K}_\alpha$.
Therefore, a point $\mathbf{z} \in \mathbb{K}_\alpha$ can be in $\mathcal{F}_d \bmod \Phi(\mathbb{Z}[\alpha])$ and $\mathcal{F}_{d'} \bmod \Phi(\mathbb{Z}[\alpha])$ for distinct $d, d' \in \mathcal{D}$ only if it is on the boundary of~$\mathcal{F}_d \bmod \Phi(\mathbb{Z}[\alpha])$.
If $\Phi(b n / \alpha^{k+1})$ lies in the interior of $\mathcal{F}_d \bmod \Phi(\mathbb{Z}[\alpha])$, then we can infer that $\varepsilon_k(n) = d$.
Thus, using the approximations $\mathcal{F}_{d,r}$ instead of~$\mathcal{F}_d$, we obtain
\begin{multline*}
S'_{k,w}(N) =  \sum_{1\le n\le N} \prod_{0\le j <|w|} \mathbf{1}_{\Phi(\mathbb{Z}[\alpha])+\mathcal{F}_{w_j,r}}\bigg( \Phi\Big(\frac{b n}{\alpha^{k+j+1}}\Big) \bigg)\\
+ \mathcal{O}\bigg(\# \bigg\{ 1 \le n \le N: \Phi\Big(\frac{b n}{\alpha^{k+j+1}}\Big) \in  (\mathcal{F}_{w_j}\, \triangle\, \mathcal{F}_{w_j,r})\cup\partial \mathcal{F}_{w_j}  \bmod \Phi(\mathbb{Z}[\alpha])\ \mbox{for some $j$}  \bigg\} \bigg),
\end{multline*}
where $A\triangle B$ denotes the symmetric difference of the sets $A$ and~$B$.
Note that
\[
\mathbf{1}_{\Phi(\mathbb{Z}[\alpha])+\mathcal{F}_{d,r}}(\mathbf{z}) = f_{d,r}(\mathbf{z}) \quad \mbox{if}\ (\mathbf{z} + D_r) \cap \big(\Phi(\mathbb{Z}[\alpha]) + \partial \mathcal{F}_{d,r}\big) = \emptyset,
\]
and $(\mathbf{z} + D_r) \cap (\Phi(\mathbb{Z}[\alpha])+\partial \mathcal{F}_{d,r}) \ne \emptyset$ implies that $\mathbf{z} \in \Phi(x) + D_r$ for some $x \in \alpha^{-r} \mathbb{Z}[\alpha]$ with $(\Phi(x) + D_r) \cap (\Phi(\mathbb{Z}[\alpha])+\partial \mathcal{F}_{d,r}) \neq \emptyset$.
Therefore, we set
\[
B_{d,r} = \big\{x \in \alpha^{-r} \mathbb{Z}[\alpha]:\, \big(\Phi(x) + D_r\big) \cap \big((\mathcal{F}_d\, \triangle\, \mathcal{F}_{d,r}) \cup \partial \mathcal{F}_d \cup \partial \mathcal{F}_{d,r}\big) \neq \emptyset \big\},
\]
i.e., $\Phi(B_{d,r}) + D_r$ forms a tube containing the boundaries of $\mathcal{F}_d$ and~$\mathcal{F}_{d,r}$.
Define
\[
F_{k,r} = \#  \bigg\{ 1 \le n \le N : \Phi\Big(\frac{b n}{\alpha^{k+1}}\Big) \in \bigcup_{d\in\mathcal{D}} \bigcup_{x\in B_{d,r}} \big(\Phi(x) + D_r\big) \bmod \Phi(\mathbb{Z}[\alpha]) \bigg\}.
\]
Then we have
\begin{align}\label{eq:3}
S'_{k,w} (N) = \sum_{1\le n\le N} \prod_{0\le j <|w|} f_{w_j,r}\bigg( \Phi\Big(\frac{b n}{\alpha^{k+j+1}}\Big) \bigg) +\mathcal{O}\Bigg( \sum_{0\le j <|w|} F_{k+j,r} \Bigg).
\end{align}

Since $B_{d,r}$ contributes to our error term, we will need the following estimate on the number of its elements.

\begin{lemma}\label{lem1}
There exists a positive constant $\varrho < a$ such that
\[
\# B_{d,r} = \mathcal{O}(\varrho^r) \qquad (d \in \mathcal{D},\ r \ge 1).
\]
\end{lemma}

\begin{proof}
First set $\mathcal{F}'_r = \bigcup_{d\in\mathcal{D}} \mathcal{F}_{d,r}$ and
\[
B'_r = \big\{x \in \alpha^{-r} \mathbb{Z}[\alpha]:\, \big(\Phi(x) + D_r\big) \cap \big((\mathcal{F}\, \triangle\, \mathcal{F}'_r) \cup \partial \mathcal{F} \cup \partial \mathcal{F}'_r\big) \neq \emptyset \big\} \qquad (r \ge 0).
\]
Since $\mathcal{F}=\alpha\mathcal{F}_d-d$, we have $B_{d,r} = \alpha B'_{r-1} - d$ for all $d \in \mathcal{D}$, $r \ge 1$, and, hence, $\# B_{d,r} = \# B'_{r-1}$. Therefore, it suffices to consider~$\# B'_r$.

By Lemma~\ref{l:tiling}, $\{\Phi(x) + \mathcal{F}:\, x \in \mathbb{Z}[\alpha]\}$ forms a tiling of~$\mathbb{K}_\alpha$, in particular $\mathcal{F}$ has non-empty interior.
As $\mathcal{F}'_r$ approximates~$\mathcal{F}$ and the diameter of~$D_r$ decreases as $r \to \infty$,  there exists some~$k$ and some $x \in \alpha^{-k} \mathbb{Z}[\alpha]$ such that $\Phi(x) + D_k$ lies in the interior of $\mathcal{F} \cap \mathcal{F}'_k$.
This implies that
\begin{equation}\label{inters}
\big(\Phi(x + \mathbb{Z}[\alpha]) + D_k\big) \cap \big((\mathcal{F}\, \triangle\, \mathcal{F}'_k) \cup \partial \mathcal{F} \cup \partial \mathcal{F}'_k\big) = \emptyset,
\end{equation}
hence $B'_k \bmod \mathbb{Z}[\alpha]$ contains at most $a^k-1$ elements.
Mod $\Phi(\mathbb{Z}[\alpha])$, we have $\alpha^k \cdot (\mathcal{F}\, \triangle\, \mathcal{F}'_{2k}) \subseteq \mathcal{F}\, \triangle\, \mathcal{F}'_k$, $\alpha^k \cdot \partial \mathcal{F} \subseteq \partial \mathcal{F}$, and $\alpha^k \cdot \partial \mathcal{F}'_{2k} \subseteq \partial \mathcal{F}'_k$, thus \eqref{inters} implies that
\[
(\Phi(x + \mathbb{Z}[\alpha]) + D_k) \cap \alpha^k \cdot ((\mathcal{F}\, \triangle\, \mathcal{F}'_{2k}) \cup \partial \mathcal{F} \cup \partial \mathcal{F}'_{2k}) = \emptyset.
\]
Therefore we have, for each $y \in \alpha^{-k} \mathbb{Z}[\alpha]$,
\[
\big(\Phi(y + \alpha^{-k} x + \mathbb{Z}[\alpha]) + D_{2k}\big) \cap \big((\mathcal{F}\, \triangle\, \mathcal{F}'_{2k}) \cup \partial \mathcal{F} \cup \partial \mathcal{F}'_{2k}\big) = \emptyset,
\]
hence $B'_{2k} \bmod \mathbb{Z}[\alpha]$ contains at most $(a^k-1)^2$ elements.
Inductively, we obtain that $B'_{nk} \bmod \mathbb{Z}[\alpha]$ contains at most $(a^k-1)^n$ elements, thus $\# (B'_r \bmod \mathbb{Z}[\alpha]) = \mathcal{O}(\varrho^r)$ with $\rho = (a^k-1)^{1/k}$.
As $\mathcal{F}$ is compact, this yields that $\# B'_r = \mathcal{O}(\varrho^r)$.
\end{proof}

Note that the optimal value for $\varrho$ in Lemma~\ref{lem1} is the spectral radius of the contact matrix; see \cite[Section~4]{ST:11}.

We proceed with an estimate of the cardinalities $F_{k,r}$ occurring in the error term in \eqref{eq:3}.

\begin{lemma}\label{lem2}
For $N \ge b^r$, we have
\[
F_{k,r} \ll N \varrho^r a^{-r} + N (a \varrho)^r \alpha^{-k} + \alpha^k\, (b \varrho)^r,
\]
where $\varrho < a$ is the same constant as in Lemma~\ref{lem1}.
\end{lemma}

\begin{proof}
Setting
\[
F_{x,k,r} = \bigg\{ 1 \le n \le N : \Phi\Big(\frac{b n}{\alpha^{k+1}}\Big) \in \Phi(x) + D_r \bmod \Phi(\mathbb{Z}[\alpha]) \bigg\},
\]
we can write
\begin{align}\label{squares:fm:local3}
F_{k,r} \le \sum_{d\in\mathcal{D}} \sum_{x\in B_{d,r}} F_{x,k,r}.
\end{align}
By Lemma~\ref{lem1}, the number of summands in \eqref{squares:fm:local3} is $\mathcal{O}(\varrho^r)$.
In what follows we show that
\begin{equation}\label{eq:2}
F_{x,k,r} \ll N a^{-r} + N a^r \alpha^{-k} + \alpha^k  b^r,
\end{equation}
which then implies the desired result.

W.l.o.g., we can assume that $b^{r-1} \mid N$ (if this does not hold, consider $\lceil N/b^{r-1} \rceil\, b^{r-1} \le 2 N$ instead).
Since $\mathbf{z} \in \Phi(x) + D_r$ implies
\[
\mathbf{z} + D_r \subseteq  (\Phi(x) + D_r) \cup (\Phi(x + \alpha^{-r}) + D_r),
\]
we have
\begin{align*}
\mathbf{1}_{\Phi(x+\mathbb{Z}[\alpha])+D_r}(\mathbf{z})&\le a^r \int_{D_r} 
(
\mathbf{1}_{\Phi(x+\mathbb{Z}[\alpha])+D_r}(\mathbf{z}+\mathbf{y})\,
+ \mathbf{1}_{\Phi(x+\alpha^{-r}+\mathbb{Z}[\alpha])+D_r}(\mathbf{z}+\mathbf{y})
)\,
\mathrm{d}\mu_\alpha(\mathbf{y})\\
&=g_{x,r}(\mathbf{z})+ g_{x+\alpha^{-r},r}(\mathbf{z}).
\end{align*}
Thus we get
\begin{align*}
F_{x,k,r} &= \sum_{1\le n\le N} \mathbf{1}_{\Phi(x+\mathbb{Z}[\alpha])+D_r} \bigg(\Phi\Big(\frac{b n}{\alpha^{k+1}}\Big)\bigg)\\
&\le \sum_{1\le n\le N} g_{x,r}\bigg(\Phi\Big(\frac{b n}{\alpha^{k+1}}\Big)\bigg)+g_{x+\alpha^{-r},r}\bigg(\Phi\Big(\frac{b n}{\alpha^{k+1}}\Big)\bigg).
\end{align*}
Setting $c^*_{x,r,\xi}= c_{x,r,\xi}+c_{x+\alpha^{-r},r,\xi}$ we have 
$g_{x,r}(\mathbf{z})+ g_{x+\alpha^{-r},r}(\mathbf{z}) = \sum_{\xi\in\mathbb{Z}[\alpha]} c^*_{x,r,\xi}\, \chi(\xi \cdot \mathbf{z})$. Using $|1-\e(x)|\le\min(2, 2\pi |x|)$, Lemma~\ref{lem3} yields
\begin{equation}\label{sun}
c^*_{x,r,\xi} = \begin{cases} \mathcal{O}\big(\min(a^{-r},\alpha^r b^{-r} \xi^{-2})\big) & \mbox{if}\ \xi \in \frac{\mathbb{Z}}{b^r}, \\ 0 & \mbox{otherwise.}\end{cases}
\end{equation}
We can write
\begin{align*}
F_{x,k,r} & \le \sum_{1\le n\le N} \sum_{\xi\in\frac{\mathbb{Z}}{b^r}} c^*_{x,r,\xi}\, \chi\bigg(\xi \cdot \Phi\Big(\frac{b n}{\alpha^{k+1}}\Big)\bigg) \\
& = \sum_{0\le n<N/b^{r-1}} \sum_{1\le m\le b^{r-1}} \sum_{\xi\in\mathbb{Z}} c^*_{x,r,\xi/b^r}\, \tilde\chi\bigg( \frac{\xi\, (nb^{r-1}+m)}{b^{r-1} \alpha^{k+1}}\bigg) \\
& =  \sum_{1\le m\le b^{r-1}} \sum_{\xi\in\mathbb{Z}} c^*_{x,r,\xi/b^r}\, \tilde\chi\bigg( \frac{\xi m}{b^{r-1} \alpha^{k+1}} \bigg) \sum_{0\le n< N/b^{r-1}} \e\!\bigg(-\frac{\xi}{\alpha^{k+1}}\, n\bigg),
\end{align*}
where we have used that $\lambda_p\big(\frac{n \xi}{\alpha^{k+1}}\big) = 0$ for each $p\in S_\alpha\setminus\{\infty\}$.
Thus we get
\begin{equation}\label{eq:1}
\begin{array}{rl}
\displaystyle F_{x,k,r} & \displaystyle  \ll \sum_{1\le m \le b^{r-1}} \sum_{\xi\in\mathbb{Z}} \big|c^*_{x,r,\xi/b^r}\big|\, \Bigg|\sum_{0\le n< N/b^{r-1}}  \e\!\bigg(-\frac{\xi}{\alpha^{k+1}}\, n\bigg)\Bigg|  \\
& \displaystyle \ll b^r \sum_{\xi\in\mathbb{Z}} \big|c^*_{x,r,\xi/b^r}\big|\, \min\!\bigg( \frac{N}{b^{r-1}}, \bigg\|\frac{\xi}{\alpha^{k+1}}\bigg\|^{-1} \bigg).
\end{array}
\end{equation}
If $\xi = 0$, then
\begin{equation}\label{eq:4}
b^r\ \big|c^*_{x,r,\xi/b^r}\big|\, \min\!\bigg( \frac{N}{b^{r-1}}, \bigg\|\frac{\xi}{\alpha^{k+1}}\bigg\|^{-1} \bigg) \ll b^r a^{-r}\, \frac{N}{b^r} = N a^{-r}.
\end{equation}
If $0 < |\xi| \le \alpha^k$, then
\[
\bigg\|\frac{\xi}{\alpha^{k+1}}\bigg\|^{-1} \le \max\!\bigg( \bigg|\frac{1}{\alpha^{k+1}}\bigg|^{-1}, \bigg\|\frac{1}{\alpha}\bigg\|^{-1} \bigg) \ll \alpha^k.
\]
The estimates in \eqref{sun} yield that
\begin{align*}
\sum_{0\le|\xi|\le a^r} \big|c^*_{x,r,\xi/b^r}\big| = \mathcal{O}(1) \quad \mbox{and} \ \sum_{|\xi|>a^r} \big|c^*_{x,r,\xi/b^r}\big| \ll \sum_{\xi>a^r} \alpha^r b^{-r} \frac{b^{2r}}{\xi^2} = \mathcal{O}(1),
\end{align*}
where we have used that\footnote{We will use this inequality several times in this work without explicitly saying so.} $\sum_{\xi>u} \frac{1}{\xi^2} \le \int_{\lfloor u\rfloor}^\infty \frac{1}{\xi^2} \mathrm{d} \xi = \frac{1}{\lfloor u\rfloor}$.
Thus we get
\begin{equation}\label{eq:6}
b^r \sum_{0<|\xi|\le\alpha^k} \big|c^*_{x,r,\xi/b^r}\big|\, \min\!\bigg( \frac{N}{b^{r-1}}, \bigg\|\frac{\xi}{\alpha^{k+1}}\bigg\|^{-1} \bigg) \ll \alpha^k\, b^r.
\end{equation}
Moreover, if $|\xi|>\alpha^k$ we have (using~\ref{sun})
\begin{equation}\label{eq:7}
b^r \sum_{|\xi|>\alpha^k} \big|c^*_{x,r,\xi/b^r}\big|\, \min\!\bigg( \frac{N}{b^{r-1}}, \bigg\|\frac{\xi}{\alpha^{k+1}}\bigg\|^{-1} \bigg) \ll N \sum_{|\xi|>\alpha^k} \alpha^r b^{-r}\, \frac{b^{2r}}{\xi^2} \ll N a^r \alpha^{-k}.
\end{equation}
Equations~\eqref{eq:4},~\eqref{eq:6}, and~\eqref{eq:7} together with~\eqref{eq:1} finally yield~\eqref{eq:2}.
\end{proof}

We are now in a position to prove our main theorem.

\begin{proof}[Proof of the main theorem]
By~\eqref{eq:8} we have
\begin{equation}\label{SwEq}
 S_w(N) = \sum_{M\le k\le L-M} S'_{k,w}(N) + \mathcal{O}(N M).
\end{equation}
Setting
\[
t_{w,k,r}(n) = \prod_{0\le j <|w|} f_{w_j,r}\bigg( \Phi\Big(\frac{b n}{\alpha^{k+j+1}}\Big) \bigg),
\]
and inserting~\eqref{eq:3} in \eqref{SwEq}, we derive
\begin{equation}\label{SwEq2}
S_w(N) = \sum_{M\le k\le L-M} \sum_{1\le n\le N} t_{w,k,r}(n) + \mathcal{O}\left(N M +  \sum_{M\le k\le L-M} \sum_{0\le j <|w|} F_{k+j,r}\right),
\end{equation}
 where the constants $M$ and $r$  will be chosen at the end of the proof. Using Lemma~\ref{lem2} ($\varrho$ is defined in Lemma~\ref{lem1}), we get
\begin{align*}
  \sum_{M\le k\le L-M} \sum_{0\le j <|w|} F_{k+j,r} &\ll  \sum_{M\le k\le L-M} (N \varrho^r a^{-r} + N (a \varrho)^r \alpha^{-k} + \alpha^k\, (b \varrho)^r)\\
&\ll LN\varrho^r a^{-r} + N (a \varrho)^r\alpha^{-M}+ \alpha^{L-M}\, (b \varrho)^r.
\end{align*}
Since $L\ll \log_{\alpha} N$ and $a>b$, we have $\alpha^{L-M}\, (b \varrho)^r \ll N (a \varrho)^r\alpha^{-M}$. Thus we may write \eqref{SwEq2} as
\begin{align}\label{eq:13}
S_w(N)  = \sum_{M\le k\le L-M} \sum_{1\le n\le N} t_{w,k,r}(n) + N\, \mathcal{O}\big(L\, \varrho^r a^{-r} + (a \varrho)^r \alpha^{-M} +  M\big).
\end{align}
Using the Fourier expansion of $f_{d,r}$, we have
\[
t_{w,k,r}(n) =  \sum_{(\xi_0,\ldots,\xi_{|w|-1})\in\mathbb{Z}^{|w|}} T_{(\xi_0,\ldots,\xi_{|w|-1})}\, \chi\Bigg( \sum_{0\le j <|w|} \frac{\xi_j}{b^r} \cdot \Phi\Big(\frac{b n}{\alpha^{k+j+1}}\Big)\Bigg),
\]
where
\[
T_{(\xi_0,\ldots,\xi_{|w|-1})} = \prod_{0\le j <|w|} c'_{w_j,r,\xi_j/b^r}.
\]
Lemma~\ref{lem6} now implies
\[
T_{(0,\ldots,0)}= a^{-|w|}
\]
and
\begin{equation}\label{star11}
\sum_{\xi_j\in\mathbb{Z}} \big|c'_{w_j,r,\xi_j/b^r}\big| =  \sum_{0\le |\xi_j|\le a^r} \big|c'_{w_j,r,\xi_j/b^r}\big| +\sum_{|\xi_j|> a^r} \big|c'_{w_j,r,\xi_j/b^r}\big| \ll a^r + \sum_{|\xi|>a^r} \alpha^{2r} \frac{b^{2r}}{\xi^2} \ll a^r
\end{equation}
for each $0\le j <|w|$. Thus we obtain
\begin{align}\label{eq:12}
\sum_{(\xi_0,\ldots,\xi_{|w|-1})\in\mathbb{Z}^{|w|}} \big|T_{(\xi_0,\ldots,\xi_{|w|-1})}\big| = \prod_{0\le j < |w|} \sum_{\xi_j\in\mathbb{Z}} \big|c'_{w_j,r,\xi_j/b^r}\big| \ll a^{|w|r}.
\end{align}
Let $E(k,r,N)$ be defined by
\[
E(k,r,N) = \sum_{(\xi_0,\ldots,\xi_{|w|-1})\in\mathbb{Z}^{|w|}\setminus\{\mathbf{0}\}} T_{(\xi_0,\ldots,\xi_{|w|-1})} \sum_{1\le n\le N} \tilde\chi\Bigg( \sum_{0\le j<|w|} \frac{\xi_j n}{b^{r-1} \alpha^{k+j+1}} \Bigg).
\]
Then we have
\[
\sum_{1\le n\le N}  t_{w,k,r}(n) = \frac{N}{a^{|w|}} + E(k,r,N).
\]
Next we bound the error term $E(k,r,N)$.
For the considerations that follow we replace $N$ by a number that is divisible by  $b^{r-1}$.
Let $N'$ and $N''$ be integers such that $N = N' + N''$, $b^{r-1} \mid N'$ and $0 \le N'' < b^{r-1}$.
Equation~\eqref{eq:12} implies
\[
E(k,r,N) = E(k,r,N') + \mathcal{O}(a^{|w|r} b^r),
\]
and we get
\begin{equation}\label{eq:9}
\sum_{1\le n\le N}  t_{w,k,r}(n) = \frac{N}{a^{|w|}} + E(k,r,N') + \mathcal{O}(a^{|w|r} b^r).
\end{equation}
The expression $E(k,r,N')$ satisfies
\begin{align*}
E(k,r,N')&  = \sum_{(\xi_0,\ldots,\xi_{|w|-1})\in\mathbb{Z}^{|w|}\setminus\{\mathbf{0}\}} T_{(\xi_0,\ldots,\xi_{|w|-1})} \sum_{1\le n\le N'} \tilde\chi\Bigg( \sum_{0\le j<|w|} \frac{\xi_j n}{b^{r-1} \alpha^{k+j+1}} \Bigg) \\
& \hspace{-3em} = \hspace{-1em} \sum_{0\le n<N'/b^{r-1}} \sum_{1\le m\le b^{r-1}} \sum_{(\xi_0,\ldots,\xi_{|w|-1})\in\mathbb{Z}^{|w|}\setminus\{\mathbf{0}\}} \hspace{-1em} T_{(\xi_0,\ldots,\xi_{|w|-1})}\, \tilde\chi\Bigg( \sum_{0\le j<|w|} \frac{\xi_j (b^{r-1}n+m)}{b^{r-1} \alpha^{k+j+1}} \Bigg)\\
&\hspace{-6em} = \hspace{-1em} \sum_{1\le m\le b^{r-1}} \hspace{-.5em} \tilde\chi\Bigg( \sum_{0\le j<|w|} \frac{\xi_j m}{b^{r-1} \alpha^{k+j+1}} \Bigg) \hspace{-.4em} \sum_{(\xi_0,\ldots,\xi_{|w|-1})\in\mathbb{Z}^{|w|}\setminus\{\mathbf{0}\}} \hspace{-2.5em} T_{(\xi_0,\ldots,\xi_{|w|-1})} \hspace{-.6em} \sum_{0\le n<N'/b^{r-1}} \hspace{-.6em} \e\!\Bigg(\!\! -\hspace{-.5em}\sum_{0\le j<|w|} \frac{\xi_j}{\alpha^{k+j+1}}\, n \Bigg),
\end{align*}
and we obtain
\begin{align}
E(k,r,N')
\ll b^r \hspace{-2em} \sum_{(\xi_0,\ldots,\xi_{|w|-1})\in\mathbb{Z}^{|w|}\setminus\{\mathbf{0}\}} \hspace{-2em}  \big|T_{(\xi_0,\ldots,\xi_{|w|-1})}\big|\, \min\!\Bigg(\frac{N'}{b^{r-1}}, \Bigg\|\sum_{0\le j <|w|}  \frac{\xi_j}{\alpha^{k+j+1}}\Bigg\|^{-1} \Bigg). \label{eq:11}
\end{align}
If $\xi_j \in a \mathbb{Z} \setminus \{0\}$ for some $0 \le j < |w|$, then Lemma~\ref{lem6} yields that
\[
T_{(\xi_0,\ldots,\xi_{|w|-1})} = 0.
\]
Thus we only have to consider vectors $(\xi_0,\ldots,\xi_{|w|-1}) \ne \mathbf0$ such that $\xi_j \not\in a \mathbb{Z} \setminus \{0\}$ for all $0\le j <|w|$. Depending on the maximal entry of the vector $(\xi_0,\ldots,\xi_{|w|-1})$, we use different estimations in order to bound $E(k,r,N')$. We have two different cases:
\begin{itemize}
\item
Assume first that $|\xi_j| \le \alpha^k/|w|$ for all $0\le j <|w|$. Then we have $\big|\sum_{0\le j <|w|}  \frac{\xi_j}{\alpha^{k+j+1}}\big| \le \frac{1}{\alpha}$. Since $\xi_j \not\in a \mathbb{Z}$ for the maximal $j < |w|$ with $\xi_j \ne 0$, we also have
\[
\Bigg|\sum_{0\le j <|w|}  \frac{\xi_j}{\alpha^{k+j+1}}\Bigg| = \frac{1}{\alpha^{k+1}}\, \Bigg|\sum_{0\le j <|w|}  \frac{\xi_j b^j}{a^j}\Bigg| \ge \frac{1}{\alpha^{k+1}}\, \frac{1}{a^{|w|-1}} \gg \alpha^{-k}.
\]
This implies
\[
\Bigg\|\sum_{0\le j <|w|}  \frac{\xi_j}{\alpha^{k+j+1}}\Bigg\|^{-1} \ll \alpha^k.
\]
\item
Assume now that $|\xi_j|>\alpha^k/|w|$ for some $0\le j< |w|$. Then
\[
 \sum_{|\xi_i|>\alpha^k/|w|} \big|c'_{w_j,r,\xi_j/b^r}\big| \ll \sum_{|\xi|>\alpha^k/|w|} \alpha^{2r} \frac{b^{2r}}{\xi^2} \ll a^{2r}\,\alpha^{-k}.
\]
Together with~\eqref{star11}, this implies
\[
\sum_{\substack{(\xi_0,\ldots,\xi_{|w|-1})\in\mathbb{Z}^{|w|}\setminus\{\mathbf{0}\}:\\ |\xi_j| > \alpha^k/|w|\ \mathrm{for}\ \mathrm{some}\ j}} \hspace{-2em}  \big|T_{(\xi_0,\ldots,\xi_{|w|-1})}\big| \ll a^{(|w|-1)r} \,a^{2r}\,\alpha^{-k}=a^{(|w|+1)r} \,\alpha^{-k}.
\]
\end{itemize}
We get, using~\eqref{eq:12} and the fact that $N'\le N$,
\begin{align*}
 E(k,r,N') & \ll b^r \hspace{-2em} \sum_{\substack{(\xi_0,\ldots,\xi_{|w|-1})\in\mathbb{Z}^{|w|}\setminus\{\mathbf{0}\}:\\ |\xi_j| \le \alpha^k/|w|\ \mathrm{for}\ \mathrm{all}\ j}} \hspace{-2em}  \big|T_{(\xi_0,\ldots,\xi_{|w|-1})}\big|\, \min\!\Bigg(\frac{N'}{b^{r-1}}, \Bigg\|\sum_{0\le j <|w|}  \frac{\xi_j}{\alpha^{k+j+1}}\Bigg\|^{-1} \Bigg)\\
&\qquad + b^r \hspace{-2em} \sum_{\substack{(\xi_0,\ldots,\xi_{|w|-1})\in\mathbb{Z}^{|w|}\setminus\{\mathbf{0}\}:\\ |\xi_j| > \alpha^k/|w|\ \mathrm{for}\ \mathrm{some}\ j}} \hspace{-2em}  \big|T_{(\xi_0,\ldots,\xi_{|w|-1})}\big|\, \min\!\Bigg(\frac{N'}{b^{r-1}}, \Bigg\|\sum_{0\le j <|w|}  \frac{\xi_j}{\alpha^{k+j+1}}\Bigg\|^{-1} \Bigg)\\
&\ll  \alpha^k\, a^{|w|r}\, b^r +  N\, a^{(|w|+1)r}\, \alpha^{-k}.
\end{align*}
Inserting this in~\eqref{eq:9} and summing over $k$ implies
\[
\sum_{M\le k\le L-M} \sum_{1\le n\le N} t_{w,k,r}(n) = \frac{N L}{a^{|w|}} + N\, \mathcal{O}\bigg(M + \frac{a^{(|w|+1)r}}{\alpha^M} + \frac{a^{|w|r}b^r }{\alpha^M} + \frac{L\, a^{|w|r} b^r}{N}\bigg).
\]
Now~\eqref{eq:13} yields
\begin{align*}
S_w(N) = \frac{N}{a^{|w|}} \log_\alpha N
+ N\, \mathcal{O}\bigg(M + \frac{L \varrho^r}{a^r} + \frac{(a \varrho)^r +  (a^{|w|+1})^r+(a^{|w|} b)^r}{\alpha^M} + \frac{L\, a^{|w|r} b^r}{N}\bigg).
\end{align*}
Let $r = \big\lfloor\frac{\log\log N}{\log(a/\varrho)}\big\rfloor$ and $M = \lfloor C \log\log N \rfloor$, with $C$ large enough such that
\[
\frac{(a \varrho)^r +  (a^{|w|+1})^r+(a^{|w|} b)^r}{\alpha^M} = \mathcal{O}(1).
\]
Then we have $S_w(N) = \frac{N}{a^{|w|}} \log_\alpha N + \mathcal{O}(N \log\log N)$, which proves the main theorem.
\end{proof}

Our two corollaries follow quite immediately from the main theorem.

\begin{proof}[Proof of Corollary~\ref{c1}]
Let the representation of $n$ in base $a/b$ be given as in \eqref{rationalrepresentation}. The summatory function  of $s_{a/b}$ satisfies
\[
\sum_{n=1}^{N}s_{a/b}(n) = \sum_{n=1}^{N}\sum_{k= 0}^{\ell(n)-1} \varepsilon_k(n) = \sum_{d\in\mathcal{D}} d S_d(N).
\]
Thus, the main theorem implies that
\[
 \sum_{n=1}^{N}s_{a/b}(n) =  \frac{N}{a} \log_{a/b} N \left(\sum_{d=0}^{a-1} d\right) + \mathcal{O}\big(N \log\log N \big),
\]
which proves the desired result.
\end{proof}

\begin{proof}[Proof of Corollary~\ref{c2}]
Let $(z_n)_{n\ge1}$ be the sequence of digits of $\mathfrak{z}_{a/b}$ in base~$a$, that is,
\[
 \mathfrak{z}_{a/b} = \sum_{n\ge1} \frac{z_n}{a^n}\,.
\]
If $w=(w_{r-1}, \ldots, w_0)$, with $w_i\in \mathcal{D}$, is a sequence of digits (of length $|w|=r$), set
\[
\gamma_w(x)= \#\,\{ 1\le n \le x : (z_{n+r-1}, \ldots, z_n)=w\}.
\]
We have to show that for each finite sequence of digits $w$ one has
\[
 \lim_{x\to\infty} \frac{\gamma_w(x)}{x} = \frac{1}{a^{|w|}}.
\]
Let $N_x$ be the largest integer satisfying
\[
 \sum_{n=1}^{N_x}\ell(n) \le x+r.
\]
Then we have $\gamma_w(x) \ge S_w(N_x)$ and $\gamma_w(x) \le S_w(N_x)  + (|w|-1)(N_x-1)+ \ell(N_x+1)$.
Hence, we obtain
$\gamma_w(x) = S_w(N_x) + \mathcal{O}(N_x)$, and the main
theorem implies
\[
 \gamma_w(x) = \frac{N_x}{a^{|w|}}\log_{a/b} N_x + \mathcal{O}\big(N_x \log\log N_x\big).
\]
Since $x= N_x\log_{a/b} N_x + O(N_x)$, we have proved that $\mathfrak{z}_{a/b}$ is a normal number in base~$a$.
\end{proof}

\section{Perspectives}\label{section:perspectives}

In this section, we want to discuss briefly some open questions and possible directions of future research related to the topic of the present paper.

\subsubsection*{Distribution of $s_{a/b}$ in residue classes}
A well-known theorem of Gelfond~\cite{Gelfond:68} states that the sum-of-digits function $s_q$ in base $q$ is equidistributed in residue classes. To be more precise, Gelfond showed that if $q,m$, and $r$ are positive integers with $q\ge2$ and $(m,q-1) = 1$, then
\begin{align*}
\#\,\{1\le n \le N: n \equiv \ell_1\bmod r,\, s_q(n) \equiv \ell_2 \bmod m\,\} = \frac{N}{mr} + \mathcal{O}(N^{\lambda})
\end{align*}
for all $\ell_1,\ell_2\in\mathbb{Z}$, where $\lambda <1$ is a positive constant only depending on $q$ and  $m$. It would be interesting to obtain similar results for  $s_{a/b}$. An easier version of this problem consists in studying the analogous problem for subsets of $\mathbb{Z}[a/b]$ rather than $\mathbb{N}$.

\subsubsection*{Rational number systems, primes and polynomials}
Mauduit and Rivat~\cite{MR:09,MR:10} recently showed that the sum-of-digits function of primes as well as squares is equidistributed in residue classes. It seems to be difficult to obtain nontrivial bounds for $\{p\le N : p \mbox{ prime, } s_{a/b}(p)\equiv \ell \bmod m\}$ and   $\{n\le N : s_{a/b}(n^2)\equiv \ell \bmod m\}$. As in the previous problem, attacking the same questions in $\mathbb{Z}[a/b]$ rather than in $\mathbb{N}$ could be more doable.

\subsubsection*{Asymptotic distribution results for $s_{a/b}$}
Bassily and K\'atai~\cite{BK:95} showed that the standard base-$q$ sum-of-digits function on polynomial sequences is asymptotic normally distributed.
Can one get results on the asymptotic behavior of $s_{a/b}$ on different subsequences using our Fourier analytic approach? Compare also with~\cite{Mad:10}, where asymptotic normality was proven for the sum-of-digits function in the Gaussian integers and in more general number systems.

\subsubsection*{Number systems in finite fields and canonical number systems}
Beck et al.~\cite{BBST:09} introduced a rather general notion of number systems defined for polynomial rings over finite fields. Here non-monic polynomials form the analogs of rational bases. Since Mahler's problem is better understood for non-monic polynomials over finite fields (see e.g.\ Allouche et al.~\cite{ADKK:01}), one can probably gain more complete results and better error terms in this setting. Moreover, the relations between these number systems and the associated Mahler problem are not yet explored here. Exploring this relation could well lead to new insights. It is also not known how difficult the underlying language of representations is.

Another possible generalization would be canonical number systems; see Peth\H{o}~\cite{Pet:91} for a definition. Here one could combine the results of Dumont et al.~\cite{DGT:99} on representations of integers in canonical number systems with our results and explore generalizations of Mahler's problem for algebraic numbers.

\bibliographystyle{amsalpha}
\bibliography{rational_patterns}
\end{document}